\documentclass[runningheads]{llncs}
\usepackage{graphicx}
\usepackage{stmaryrd} 
\usepackage{multirow} 
\usepackage{amsmath}
\usepackage{amssymb}

\newcommand{\lbvarphi}{\underline{\varphi}}
\newcommand{\ubvarphi}{\overline{\varphi}}
\newcommand{\lblambda}{\underline{\lambda}}
\newcommand{\ublambda}{\overline{\lambda}}
\newcommand{\kcivarphi}{k_{ij}}
\newcommand{\kcilambda}{l_{ij}}
\newcommand{\fspace}{\;\;}

\def\argmin{\mathop{\rm argmin}}
\def\ri{\mathop{\rm ri}}

\newtheorem{observation}[theorem]{Observation}
\newtheorem{consequence}[theorem]{Consequence}
\newtheorem{fact}[theorem]{Fact}

\newcommand{\iverson}[1]{\llbracket#1\rrbracket}

\begin{document}
\title{A Class of Linear Programs Solvable by Coordinate-Wise Minimization\thanks{This is a pre-print of the paper {\textit{Dlask, T., Werner, T.: A class of linear programs solvable by coordinate-wise minimization. In: Kotsireas, I.S., Pardalos, P.M. (eds.) LION 2020. LNCS, vol. 12096, pp. 52-67. Springer, Cham (2020).}} The final authenticated version is available online at https://doi.org/10.1007/978-3-030-53552-0\_8. \hfill \protect\\  \protect\\ This work has been supported by the Grant Agency of the Czech Technical University in Prague (grant SGS19/170/OHK3/3T/13), the OP VVV project CZ.02.1.01/0.0/0.0/16\textunderscore019/0000765, and the Czech Science Foundation (grant 19-09967S).}}
%
%
\author{Tom\'{a}\v{s} Dlask\orcidID{0000-0002-1944-6569} \and
Tom\'{a}\v{s} Werner\orcidID{0000-0002-6161-7157}}
\authorrunning{T. Dlask and T. Werner}
%
\institute{Faculty of Electrical Engineering, Czech Technical University in Prague\\
\email{dlaskto2@fel.cvut.cz}}
\maketitle              
\begin{abstract}
Coordinate-wise minimization is a simple popular method for
large-scale optimization. Unfortunately, for general
(non-differentiable) convex problems it may not find global minima. We
present a class of linear programs that coordinate-wise minimization
solves exactly. We show that dual LP relaxations of several well-known
combinatorial optimization problems are in this class and the method
finds a global minimum with sufficient accuracy in reasonable
runtimes. Moreover, for extensions of these problems that no longer are
in this class the method yields reasonably good suboptima. Though the
presented LP relaxations can be solved by more efficient methods (such
as max-flow), our results are theoretically non-trivial and can lead
to new large-scale optimization algorithms in the future.

\end{abstract}
\section{Introduction}

{\em Coordinate-wise minimization\/}, or {\em coordinate descent\/},
is an iterative optimization method, which in every iteration
optimizes only over a single chosen variable while keeping the
remaining variables fixed. Due its simplicity, this method is popular
among practitioners in large-scale optimization in areas such as
machine learning or computer vision, see e.g.~\cite{Wright15}.  A
natural extension of the method is {\em block-coordinate
minimization\/}, where every iteration minimizes the objective over a
block of variables. In this paper, we focus on coordinate minimization
with exact updates, where in each iteration a global minimum over the
chosen variable is found, applied to convex optimization problems.

For general convex optimization problems, the method need not converge
and/or its fixed points need not be global minima. A simple example is
the unconstrained minimization of the function
$f(x,y)=\max\{x-2y,y-2x\}$, which is unbounded but any point with
$x=y$ is a coordinate-wise local minimum. Despite this drawback,
(block-)coordinate minimization can be very successful for some
large-scale convex non-differentiable problems. The prominent example
is the class of {\em convergent message passing\/} methods for solving
dual linear programming (LP) relaxation of maximum a posteriori (MAP)
inference in graphical models, which can be seen as various forms of
(block-)coordinate descent applied to various forms of the dual. In
the typical case, the dual LP relaxation boils down to the
unconstrained minimization of a convex piece-wise affine (hence
non-differentiable) function.  These methods include max-sum diffusion
\cite{Kovalevsky-diffusion,Schlesinger-2011,Werner-PAMI07},
TRW-S~\cite{Kolmogorov06}, MPLP~\cite{Globerson08}, and
SRMP~\cite{Kolmogorov-PAMI-2015}. They do not guarantee global
optimality but for large sparse instances from computer vision the
achieved coordinate-wise local optima are very good and TRW-S is
significantly faster than competing
methods~\cite{Szeliski:study:PAMI:2008,Kappes-study-2015}, including
popular first-order primal-dual methods such as ADMM~\cite{Boyd:2011}
or~\cite{ChambollePock-2011}.

This is a motivation to look for other classes of convex optimization
problems for which (block-)coordinate descent would work well or,
alternatively, to extend convergent message passing methods to a wider
class of convex problems than the dual LP relaxation of MAP
inference. A step in this direction is the
work~\cite{Prusa-Werner-arXiv-relint}, where it was observed that if
the minimizer of the problem over the current variable block is not
unique, one should choose a minimizer that lies in the {\em relative
interior\/} of the set of block-optimizers. It is shown that any
update satisfying this rule is, in a precise sense, not worse than any
other exact update. Message-passing methods such as max-sum diffusion
and TRW-S satisfy this rule. If max-sum diffusion is modified to
violate the relative interior rule, it can quickly get stuck in a very
poor coordinate-wise local minimum.

To be precise, suppose we minimize a convex function
$f{:}\ X\to\mathbb{R}$ on a closed convex set
$X\subseteq\mathbb{R}^n$. We assume that $f$~is bounded from below
on~$X$. For brevity of formulation, we rephrase this as the
minimization of the extended-valued function
$\bar f{:}\ \mathbb{R}^n\to\mathbb{R}\cup\{\infty\}$ such that
$\bar f(x)=f(x)$ for $x\in X$ and $\bar f(x)=\infty$ for $x\notin X$.
One iteration of coordinate minimization with the relative interior
rule~\cite{Prusa-Werner-arXiv-relint} chooses a variable index
$i\in[n]=\{1,\ldots,n\}$ and replaces an estimate
$x^k=(x_1^k,\ldots,x_n^k)\in X$ with a new estimate
$x^{k+1}=(x^{k+1}_1,\ldots,x^{k+1}_n)\in X$ such
that\footnote{In~\cite{Prusa-Werner-arXiv-relint}, the iteration is
formulated in a more abstract (coordinate-free) notation. Since we
focus only on coordinate-wise minimization here, we use a more
concrete notation.}
\begin{align*}
x_i^{k+1} &\in \ri \argmin_{y\in\mathbb{R}} \bar f(x^k_1,\ldots,x^k_{i-1},y,x^k_{i+1},\ldots,x^k_n) ,\\
x_j^{k+1} &= x_j^k \quad \forall j\neq i ,
\end{align*}
where $\ri Y$ denotes the relative interior of a convex set~$Y$. As
this is a univariate convex problem, the set
$Y=\argmin_{y\in\mathbb{R}} \bar f(x^k_1,\ldots,x^k_{i-1},y,x^k_{i+1},\ldots,x^k_n)$
is either a singleton or an interval. In the latter case, the relative
interior rule requires that we choose~$x_i^{k+1}$ from the interior of
this interval. A point $x=(x_1,\ldots,x_n)\in X$ that satisfies
\[
x_i \in \ri \argmin_{y\in\mathbb{R}} \bar f(x_1,\ldots,x_{i-1},y,x_{i+1},\ldots,x_n)
\]
for all $i\in[n]$ is called a (coordinate-wise) {\em interior local
minimum\/} of function~$f$ on set~$X$.

Some classes of convex problems are solved by coordinate-wise
minimization exactly. E.g., for unconstrained minimization of a
differentiable convex function, it is easy to see that any fixed point
of the method is a global minimum; moreover, it has been proved that
if the function has unique univariate minima, then any limit point is
a global minimum \cite[\S2.7]{Bertsekas99}. The same properties hold
for convex functions whose non-differentiable part is
separable~\cite{Tseng:2001}. Note that these classical results need
not assume the relative interior
rule~\cite{Prusa-Werner-arXiv-relint}.

Therefore, it is natural to ask if the relative interior rule can
widen the class of convex optimization problems that are exactly
solved by coordinate-wise minimization. Leaving convergence
aside\footnote{We do not discuss convergence in this paper and assume
that the method converges to an interior local minimum. This is
supported by experiments, e.g., max-sum diffusion and TRW-S have this
property. More on convergence can be found
in~\cite{Prusa-Werner-arXiv-relint}.}, more precisely we can ask for
which problems interior local minima are global minima. A succinct
characterization of this class is currently out of reach. Two
subclasses of this class are known
\cite{Kolmogorov06,Schlesinger-2011,Werner-PAMI07}: the dual LP
relaxation of MAP inference with pairwise potential functions and two
labels, or with submodular potential functions.

In this paper, we restrict ourselves to linear programs (where $f$~is
linear and $X$~is a convex polyhedron) and present a new class of
linear programs with this property. We show that dual LP relaxations
of a number of combinatorial optimization problems belong to this
class and coordinate-wise minimization converges in reasonable time on
large practical instances. Unfortunately, the practical impact of this
result is limited because there exist more efficient algorithms for
solving these LP relaxations, such as reduction to max-flow. It is
open whether there exist some useful classes of convex problems that
are exactly solvable by (block-)coordinate descent but not solvable by
more efficient methods. There is a possibility that our result and the
proof technique will pave the way to such results.

\section{Reformulations of Problems}
\label{sec:reforms}

Before presenting our main result, we make an important remark: while
a convex optimization problem can be reformulated in many ways to an
`equivalent' problem which has the same global minima, not all of
these transformations are equivalent with respect to coordinate-wise
minimization, in particular, not all preserve interior local minima.

\begin{example}
One example is dualization. If coordinate-wise minimization achieves
good local (or even global) minima on a convex problem, it can get
stuck in very poor local minima if applied to its dual. Indeed, trying
to apply (block-)coordinate minimization to the {\em primal\/} LP
relaxation of MAP inference (linear optimization over the local
marginal polytope) has been futile so far.
\end{example}

\begin{example}\label{ex:redundant_constraint}
Consider the linear program $\min \{ x_1 + x_2 \mid x_1, x_2 \geq 0\}$,
which has one interior local minimum with respect to individual coordinates that also corresponds to the unique global optimum. But if one adds a redundant constraint, namely $x_1 = x_2$, then any feasible point will become an interior local minimum w.r.t. individual coordinates, because the redundant constraint blocks changing the variable $x_i$ without changing $x_{3-i}$ for both $i \in \{1,2\}$.
\end{example}

\begin{example}\label{ex:sum_of_maxima}
Consider the linear program
\begin{subequations} \label{eq:example_formulation1}
\begin{align}
\min & \, \sum_{j = 1}^m z_j \span\span\\
z_j & \geq a_{ij}^Tx + b_{ij} & \forall & i \in [n], j \in [m] \\
z & \in \mathbb{R}^m, x \in \mathbb{R}^p\span\span 
\end{align}
\end{subequations}
which can be also formulated as
\begin{subequations} \label{eq:example_formulation2}
\begin{align}
\min & \, \sum_{j  = 1}^m \max_{i =1}^n (a_{ij}^Tx + b_{ij}) \\
x & \in \mathbb{R}^p.
\end{align}
\end{subequations}
Optimizing over the individual variables by coordinate-wise minimization in \eqref{eq:example_formulation1} does not yield the same interior local optima as in \eqref{eq:example_formulation2}. For instance, assume that $m = 3$, $n = p = 1$ and the problem \eqref{eq:example_formulation2} is given as
\begin{equation}\label{example_minmaxform}
\min \, \left( \max\{x,0 \} + \max\{-x,-1 \} + \max\{-x,-2 \} \right),
\end{equation}
where $x \in \mathbb{R}$. Then, when optimizing directly in form \eqref{example_minmaxform}, one can see that all the interior local optima are global optimizers.

However, when one introduces the variables $z \in \mathbb{R}^3$ and applies coordinate-wise minimization on the corresponding problem \eqref{eq:example_formulation1}, then there are interior local optima that are not global optimizers, for example $x = z_1 = z_2 = z_3 = 0$, which is an interior local optimum, but is not a global optimum.

On the other hand, optimizing over blocks of variables $\{z_1, \dots, z_{m}, x_i\}$ for each $i \in [p]$ in case \eqref{eq:example_formulation1} is equivalent to optimization over individual $x_i$ in formulation \eqref{eq:example_formulation2}. 
\end{example}

\section{Main Result}

The optimization problem with which we are going to deal is in its most general form defined as
\begin{subequations} \label{eq:primal_simplified}
\begin{align}
\min \; \; & \Bigl(\sum_{i = 1}^m \max \{w_i-\varphi_i, 0\} + a^T\varphi + b^T\lambda + \sum_{j = 1}^p \max \{v_j + A_{:j}^T\varphi + B_{:j}^T\lambda,0\} \Bigr) \label{eq:primal_simplified_criterion}\\
&\lbvarphi_i \leq \varphi_i \leq \ubvarphi_i \fspace \forall i \in [m] \label{eq:primal_simplified:box1}\\ 
& \lblambda_i \leq \lambda_i \leq \ublambda_i \fspace \forall i \in [n], \label{eq:primal_simplified:box2}
\end{align}
\end{subequations}
where $A \in \mathbb{R}^{m \times p}, B \in \mathbb{R}^{n \times p}$, $a \in \mathbb{R}^{m}$, $b \in \mathbb{R}^{n}$, $w \in \mathbb{R}^m$, $v \in \mathbb{R}^p$, $\lbvarphi \in (\mathbb{R} \cup \{-\infty\})^m$, $\ubvarphi \in (\mathbb{R} \cup \{\infty\})^m$, $\lblambda \in (\mathbb{R} \cup \{-\infty\})^n$, $\ublambda \in (\mathbb{R} \cup \{\infty\})^n$ (assuming $\lbvarphi < \ubvarphi$ and $\lblambda < \ublambda$). We optimize over variables $\varphi \in \mathbb{R}^m$ and $\lambda \in \mathbb{R}^n$. $A_{:j}$ and $A_{i:}$ denotes the $j$-th column and $i$-th row of~$A$, respectively.

Applying coordinate-wise minimization with relative-interior rule on the problem \eqref{eq:primal_simplified} corresponds to cyclic updates of variables, where each update corresponds to finding the region of optima of a convex piecewise-affine function of one variable on an interval. If the set of optimizers is a singleton, then the update is straightforward. If the set of optimizers is a bounded interval $[a,b]$, the variable is assigned the middle value from this interval, i.e. $(a+b)/2$. If the set of optima is unbounded, i.e. $[a,\infty)$, then we set the variable to the value $a+\Delta$, where $\Delta > 0$ is a fixed constant. In case of $(-\infty,a]$, the variable is updated to $a-\Delta$. The details for the update in this setting are in Appendix \ref{ap:details_updates}.

\begin{theorem}\label{th:theorem_optimality}
Any interior local optimum of \eqref{eq:primal_simplified} w.r.t. individual coordinates is its global optimum if
\begin{itemize}
\setlength\itemsep{0em}
\item matrices $A,B$ contain only values from the set $\{-1,0,1\}$ and contain at most two non-zero elements per row
\item  vector $a$ contains only elements from the set $(-\infty,-2] \cup \{-1,0,1,2\} \cup [3, \infty)$
\item vector $b$ contains only elements from the set $(-\infty,-2] \cup \{-1,0,1\} \cup [2, \infty)$.
\end{itemize}
\end{theorem}

In order to prove Theorem~\ref{th:theorem_optimality}, we formulate
problem~\eqref{eq:primal_simplified} as a linear program by
introducing additional variables $\alpha \in \mathbb{R}^m$ and
$\beta \in \mathbb{R}^p$ and construct its dual. The proof of
optimality is then obtained (see
Theorem~\ref{th:theorem_dual_solution}) by constructing a dual
feasible solution that satisfies complementary slackness.

The primal linear program (with corresponding dual variables and constraints on the same lines) reads

\begin{subequations} \label{eq:LP}
\begin{align}
\min \sum_{i \in [m]} \alpha_i + \sum_{i \in [p]}&\beta_i + a^T\varphi + b^T\lambda &  \max \, f(z,y,s,r,&q,x) &&   \\
\beta_j - A_{:j}^T\varphi - B_{:j}^T\lambda & \geq v_j & x_j & \geq 0 & \forall & j \in [p] \label{eq:LP_c1}\\
\alpha_i + \varphi_i &\geq w_i & s_i & \geq 0 & \forall & i \in [m] \label{eq:LP_c2}\\
\varphi_i &\geq \lbvarphi_i & y_i & \geq 0  & \forall & i \in [m] \label{eq:LP_c3}\\ 
\varphi_i & \leq \ubvarphi_i & z_i & \leq 0  & \forall & i \in [m] \label{eq:LP_c4}\\ 
\lambda_i & \geq \lblambda_i & q_i & \geq 0  & \forall & i \in [n] \label{eq:LP_c5}\\
\lambda_i & \leq \ublambda_i & r_i & \leq 0 & \forall & i \in [n] \label{eq:LP_c6}\\
\varphi_i &\in \mathbb{R} & s_i + z_i + y_i - A_{i:}^Tx &= a_i & \forall & i \in [m] \label{eq:LP_c7}\\
\lambda_i &\in \mathbb{R} & r_i + q_i - B_{i:}^Tx &= b_i & \forall & i \in [n] \label{eq:LP_c8}\\
\beta_j & \geq 0 & x_j &\leq 1 & \forall & j \in [p] \label{eq:LP_c9}\\
\alpha_i & \geq 0 & s_i &\leq 1 & \forall & i \in [m], \label{eq:LP_c10}
\end{align}
\end{subequations}
where the dual criterion is
\begin{equation}
f(z,y,s,r,q,x) = {\ubvarphi}^Tz+{\lbvarphi}^Ty+w^Ts+{\ublambda}^Tr+{\lblambda}^Tq+v^Tx
\end{equation}
and clearly, at optimum of the primal, we have
\begin{subequations}\label{eq:additional_variable_values}
\begin{align}
\alpha_i & = \max\{w_i-\varphi_i,0\} & &\forall i \in [m] \\
\beta_j & = \max\{v_j + A_{:j}^T\varphi + B_{:j}^T\lambda,0\} & &\forall j \in [p].
\end{align}
\end{subequations}
The variables $\alpha,\beta$ were eliminated from the primal formulation \eqref{eq:LP} to obtain \eqref{eq:primal_simplified} due to similar reasoning as in Example \ref{ex:sum_of_maxima}. We also remark that setting $\ubvarphi_i = \infty$ (resp. $\lbvarphi_i = -\infty$, $\ublambda_i = \infty$, $\lblambda_i = -\infty$) results in $z_i = 0$ (resp. $y_i = 0$, $r_i = 0$, $q_i = 0$).

Even though the primal-dual pair \eqref{eq:LP} might seem
overcomplicated, such general description is in fact necessary because
as described in~\S\ref{sec:reforms}, equivalent reformulations may not
preserve the structure of interior local minima and we would like to
describe as general class, where optimality is guaranteed, as
possible.

\begin{example}
To give the reader better insight into the problems \eqref{eq:LP}, we present a simplification based on omitting the matrix $A$ (i.e. $m = 0$) and setting $\lblambda = 0$, $\ublambda = \infty$, which results in $r_i = 0$ and variables $q_i$ become slack variables in \eqref{eq:LP_c8}. The primal-dual pair in this case then simplifies to
\begin{subequations}
\begin{align}
\min \sum_{i \in [p]}&\beta_i + b^T\lambda &  \max \, &v^Tx  &  \\
\beta_j - B_{:j}^T\lambda & \geq v_j & x_j & \geq 0 & \forall & j \in [p] \\
\beta_j & \geq 0 & x_j &\leq 1 & \forall & j \in [p] \\
\lambda_i & \geq 0 & - B_{i:}^Tx &\leq b_i & \forall & i \in [n].
\end{align}
\end{subequations}
\end{example}

\begin{theorem}\label{th:theorem_dual_solution}
For a problem \eqref{eq:primal_simplified} satisfying conditions of Theorem \ref{th:theorem_optimality} and a given interior local minimum $(\varphi, \lambda)$, the values\footnote{We define $h_{[x,y]}(z) = \min\{y, \max\{z, x\}\}$ to be the projection of $z \in \mathbb{R}$ onto the interval $[x,y] \subseteq \mathbb{R}$. The projection onto unbounded intervals $(-\infty,0]$ and $[0,\infty)$ is defined similarly and is denoted by $h_{\mathbb{R}^-_0}$ and $h_{\mathbb{R}^+_0}$ for brevity.}
\begin{subequations}\label{eq:dual_solution}
\begin{align*}
x_j &= 
  \begin{cases}
    0 & \text{if } A_{:j}^T\varphi + B_{:j}^T\lambda + v_j < 0 \\
    \frac{1}{2} & \text{if } A_{:j}^T\varphi + B_{:j}^T\lambda + v_j = 0 \\
    1 & \text{if } A_{:j}^T\varphi + B_{:j}^T\lambda + v_j > 0
  \end{cases} &
s_i &= 
  \begin{cases}
    1 & \text{if } w_i>\varphi_i \\
    0 & \text{if } w_i<\varphi_i \\
    h_{[0,1]}(a_i+A_{i:}^Tx) & \text{if } w_i=\varphi_i
  \end{cases} \\
r_i &= 
  \begin{cases}
    0 & \text{if } \lambda_i < \ublambda_i \\
    h_{\mathbb{R}^-_0}(b_i + B_{i:}^Tx) & \text{if } \lambda_i = \ublambda_c
  \end{cases} & 
  z_i &= 
  \begin{cases}
    0 & \text{if } \varphi_i < \ubvarphi_i \\
    h_{\mathbb{R}^-_0}(a_i + A_{i:}^Tx - s_i) & \text{if } \varphi_i = \ubvarphi_i
  \end{cases} \\
q_i &= 
  \begin{cases}
    0 & \text{if } \lambda_i > \lblambda_i \\
    h_{\mathbb{R}^+_0}(b_i + B_{i:}^Tx) & \text{if } \lambda_i = \lblambda_i
  \end{cases} &
y_i &= 
  \begin{cases}
    0 & \text{if } \varphi_i > \lbvarphi_i \\
    h_{\mathbb{R}^+_0}(a_i + A_{i:}^Tx - s_i) & \text{if } \varphi_i = \lbvarphi_i
  \end{cases}
\end{align*}
\end{subequations}
are feasible for the dual \eqref{eq:LP} and satisfy complementary slackness with primal \eqref{eq:LP}, where the remaining variables of the primal are given by \eqref{eq:additional_variable_values}.
\end{theorem}

It can be immediately seen that all the constraints of dual 
\eqref{eq:LP} are satisfied except for \eqref{eq:LP_c7} and
\eqref{eq:LP_c8}, which require a more involved analysis. The complete proof of Theorem \ref{th:theorem_dual_solution} is technical (based on verifying many different cases) and given in Appendix \ref{ap:proof_theorem}.

\section{Applications}

Here we show that several LP relaxations of combinatorial problems
correspond to the form~\eqref{eq:primal_simplified} or to the dual \eqref{eq:LP} and
discuss which additional constraints correspond to the assumptions of
Theorem~\ref{th:theorem_optimality}.

\subsection{Weighted Partial Max-SAT}

In weighted partial Max-SAT,
one is given two sets of clauses, soft
and hard. Each soft clause is assigned a positive weight. The task is
to find values of binary variables $x_i \in \{0,1\}$, $i\in [p]$ such
that all the hard clauses are satisfied and the sum of weights of the satisfied 
soft clauses is maximized.

We organize the $m$ soft clauses into a matrix $S\in\{-1,0,1\}^{m\times p}$ defined as
\begin{equation*}
S_{ci} = \begin{cases} 
      1 & \text{if literal } x_i \text{ is present in soft clause } c\\
      -1 & \text{if literal } \neg x_i \text{ is present in soft clause } c\\
      0 & \text{otherwise }
   \end{cases},
\end{equation*}
In addition, we denote $n^S_c=\sum_i \iverson{S_{ci} < 0}$ to be the
number of negated variables in clause~$c$. These numbers are stacked
in a vector $n^S\in\mathbb{Z}^m$. The $h$~hard clauses are organized
in a matrix $H\in\{-1,0,1\}^{h\times p}$ and a vector
$n^H\in\mathbb{Z}^h$ in the same manner.

The LP relaxation of this problem reads
\begin{subequations}\label{eq:maxsat}
\begin{align}
\max & \sum_{c \in [m]} w_c s_c \span\span\\
s_c &\leq S_{c:}^Tx + n^S_c & \forall & c \in [m] \label{eq:maxsat_soft_clause}\\
H_{c:}^Tx + n^H_c & \geq 1 & \forall &  c \in [h] \label{eq:maxsat_hard_clause}\\
x_i & \in [0,1] & \forall &  i \in [p] \\
s_c & \in [0,1] & \forall &  c \in [m],
\end{align}
\end{subequations}
where $w_c \in \mathbb{R}^+_0$ are the weights of the soft clauses
$c \in [m]$. This is a sub-class of the dual \eqref{eq:LP}, where
$A = S$, $B = -H$, $a = n^S$, $b = 1-n^H$, $\lbvarphi = 0$ ($y \geq 0$
are therefore slack variables for the dual constraint \eqref{eq:LP_c7} that
correspond to \eqref{eq:maxsat_soft_clause}), $\ubvarphi = \infty$
(therefore $z = 0$), $\lblambda = -\infty$ (therefore $q = 0$),
$\ublambda = 0$ ($r \leq 0$ are slack variables for the dual constraint 
\eqref{eq:LP_c8} that correspond to
\eqref{eq:maxsat_hard_clause}), $v = 0$.

Formulation \eqref{eq:maxsat} satisfies the conditions of Theorem \ref{th:theorem_optimality} if each of the clauses has length at most 2. In other words, optimality is guaranteed for weighted partial Max-2SAT.

Also notice that if we omitted the soft clauses \eqref{eq:maxsat_soft_clause} and instead set $v = -1$, we would obtain an instance of Min-Ones SAT, which could be generalized to weighted Min-Ones SAT. This relaxation would still satisfy the requirements of Theorem \ref{th:theorem_optimality} if all the present hard clauses have length at most 2.

\subsubsection{Results}

We tested the method on 800 smallest\footnote{Smallest in the sense of
the file size. All instances could not have been evaluated due to
their size and lengthy evaluation.} instances that appeared in Max-SAT
Evaluations \cite{maxsatevaluations}
in years 2017 \cite{ansotegui2017maxsat} and 2018
\cite{bacchus2018maxsat}. The results on the instances are divided
into groups in Table \ref{ta:table_maxsat} based on the minimal and
maximal length of present clauses. We have also tested this approach
on 60 instances of weighted Max-2SAT from Ke Xu~\cite{xu2003many}. The highest number
of logical variables in an instance was 19034 and the highest overall
number of clauses in an instance was 31450.  It was important to
separate the instances without unit clauses (i.e. clauses of length
1), because in such cases the LP relaxation \eqref{eq:maxsat} has a
trivial optimal solution with $x_i = \frac12$ for all $i \in V$.

Coordinate-wise minimization was stopped when the criterion did not improve by at least $\epsilon = 10^{-7}$ after a whole cycle of updates for all variables. We report the quality of the solution as the median and mean relative difference between the optimal criterion and the criterion reached by coordinate-wise minimization before termination.

Table~\ref{ta:table_maxsat} reports not only instances of
weighted partial Max-2SAT but also instances with longer clauses,
where optimality is no longer guaranteed. Nevertheless, the relative
differences on instances with longer clauses still seem not too
large and could be usable as bounds in a branch-and-bound scheme.

\begin{table}
\centering
{\footnotesize
\begin{tabular}{ |c|c|c|c|c|c|c|c|  }
 \hline
 \multicolumn{3}{|c|}{Instance Group Specification} & \multicolumn{2}{|c|}{Results} \\
 \hline
 Min CL & Max CL & \#inst. & Mean RD & Median RD\\
 \hline 
 $\geq 2$   & any    & 91 &   0 & 0\\
1   & 2    & 123 &   $1.44 \cdot 10^{-9}$ & $1.09 \cdot 10^{-11}$\\
  1   & 3    & 99 &  $6.98 \cdot 10^{-3}$ & $1.90 \cdot 10^{-7}$\\
  1 & $\geq 4$ & 487 & $1.26 \cdot 10^{-2}$ & $2.97 \cdot 10^{-3}$ \\ \hline 
 1& 2& 60 & $1.59 \cdot 10^{-9}$ & $5.34 \cdot 10^{-10}$ \\
 \hline
\end{tabular}
}
\caption{Experimental comparison of coordinate-wise minimization and
exact solutions for LP relaxation on instances from \cite{maxsatevaluations} 
(first 4 rows) and ~\cite{xu2003many} (last row).} \label{ta:table_maxsat}
\end{table}

\subsection{Weighted Vertex Cover}\label{se:vertex_cover}

Dual \eqref{eq:LP} also subsumes\footnote{It is only necessary to transform minimization to maximization of negated objective in \eqref{eq:vertex_cover}.} the LP relaxation of weighted vertex cover, which reads
\begin{equation}\label{eq:vertex_cover}
\min \, \Big\{ \sum_{i \in V} v_i x_i \Bigm\vert x_i + x_j \geq 1, \forall \{i,j\} \in E, x_i \in [0,1], \forall i \in V \Big\}
\end{equation}
where $V$ is the set of nodes and $E$ is the set of edges of an undirected graph. This problem also satisfies the conditions of Theorem \ref{th:theorem_optimality} and therefore the corresponding primal \eqref{eq:primal_simplified} will have no non-optimal interior local minima.

On the other hand, notice that formulation \eqref{eq:vertex_cover}, which corresponds to dual  \eqref{eq:LP} can have non-optimal interior local minima even with respect to all subsets of variables of size $|V|-1$, an example is given in Appendix \ref{ap:example_weighted_vertex_cover}.

We reported the experiments on weighted vertex cover in an unpublished
text \cite{CVPR_submission} where the optimality was not proven yet. In
addition, the update designed in \cite{CVPR_submission} {\em ad hoc\/}
becomes just a special case of our general update here.

\subsection{Minimum $st$-Cut, Maximum Flow}\label{sec:mincutmaxflow_coorddesc}

Recall from \cite{fulkerson1962flows} the usual formulation of max-flow problem between nodes $s \in V$ and $t \in V$ on a directed graph with vertex set $V$, edge set $E$ and positive edge weights $w_{ij} \in \mathbb{R}^+_0$ for each $(i,j) \in E$, which reads
\begin{subequations} \label{eq:textbook_formulation_maxflow}
\begin{align}
\max \sum_{(s,i) \in E}& f_{si} \\
0 \leq f_{ij} &\leq w_{ij} & \forall &(i,j) \in E \\
\sum_{(u,i) \in E} f_{ui} - \sum_{(j,u) \in E} f_{ju} &= 0 & \forall & u \in V-\{s,t\}. \label{eq:flow_conservation}
\end{align}
\end{subequations}
Assume that there is no edge $(s,t)$, there are no ingoing edges to $s$ and no outgoing edges from $t$, then any feasible value of $f$ in \eqref{eq:textbook_formulation_maxflow} is an interior local optimum w.r.t. individual coordinates by the same reasoning as in Example \ref{ex:redundant_constraint} due to the flow conservation constraint \eqref{eq:flow_conservation}, which limits each individual variable to a single value. We are going to propose a formulation which has no non-globally optimal interior local optima.

The dual problem to \eqref{eq:textbook_formulation_maxflow} is the minimum $st$-cut problem, which can be formulated as
\begin{subequations}\label{eq:mincut}
\begin{align}
\max \, & \sum_{(i,j) \in E} w_{ij} y_{ij} \\
y_{ij} & \leq 1-x_i+x_j & \forall & (i,j) \in E, i \neq s, j\neq t \\
y_{sj} & \leq x_j & \forall & (s,j) \in E \\
y_{it} & \leq 1-x_i & \forall & (i,t) \in E \\
y_{ij} & \in [0,1] & \forall & (i,j) \in E, \\
x_i & \in [0,1] & \forall & i \in V-\{s,t\},
\end{align}
\end{subequations}
where $y_{ij} = 0$ if edge $(i,j)$ is in the cut and $y_{ij} = 1$ if edge $(i,j)$ is not in the cut. The cut should separate $s$ and $t$, so the set of nodes connected to $s$ after the cut will be denoted by $S$ and $T = V - S$ is the set of nodes connected to $t$. Using this notation, $x_i = \iverson{i \in S}$. Formulation \eqref{eq:mincut} is different from the usual formulation by replacing the variables $y_{ij}$ by $1-y_{ij}$, therefore we also maximize the weight of the not cut edges instead of minimizing the weight of the cut edges, therefore if the optimal value of \eqref{eq:mincut} is $O$, then the value of the minimum $st$-cut equals $\sum_{(i,j)\in E}w_{ij} - O$.

Formulation \eqref{eq:mincut} is subsumed by the dual \eqref{eq:LP} by setting $\lbvarphi = 0$, $\ubvarphi = \infty$ and omitting the $B$ matrix. Also notice that each $y_{ij}$ variable occurs in at most one constraint. The problem \eqref{eq:mincut} therefore satisfies the conditions of Theorem \ref{th:theorem_optimality} and the corresponding primal \eqref{eq:primal_simplified} is a formulation of the maximum flow problem, in which one can search for the maximum flow by coordinate-wise minimization. The corresponding formulation \eqref{eq:primal_simplified} reads
\begin{subequations}\label{eq:maxflow_our}
\begin{align}
\min & \Bigl(\sum_{(i,j) \in E} \max \{w_{ij}-\varphi_{ij}, 0\} + \sum_{(i,j) \in E, i \neq s}\varphi_{ij} +\nonumber \\
& + \sum_{i \in V - \{s,t\}}\max \Bigl\{\sum_{(j,i) \in E} \varphi_{ji} - \sum_{(i,j) \in E} \varphi_{ij},0\Bigr\}\Bigr) \\
& \varphi_{ij} \geq 0 \fspace \forall (i,j) \in E.
\end{align}
\end{subequations}

\subsubsection{Results}

We have tested our formulation for coordinate-wise minimization on max-flow instances\footnote{Available at \url{https://vision.cs.uwaterloo.ca/data/maxflow}.} from computer vision. We report the same statistics as with Max-SAT in Table \ref{ta:table_maxflow}, the instances corresponded to stereo problems, multiview reconstruction instances and shape fitting problems. 

For multiview reconstruction and shape fitting, we were able to run our algorithm only on small instances, which have approximately between $8 \cdot 10^5$ and $1.2 \cdot 10^6$ nodes and between $5 \cdot 10^6$ and $6 \cdot 10^6$ edges. On these instances, the algorithm terminated with the reported precision in 13 to 34 minutes on a laptop.

\begin{table}
\centering
{\footnotesize
\begin{tabular}{ |l|c|c|c|c|c|c|c|  }
 \hline
 \multicolumn{2}{|c|}{Instance Group or Instance} & \multicolumn{2}{|c|}{Results} \\
 \hline
 Name & \!\!\#inst.\!\! & Mean RD & Median RD\\
 \hline
 BVZ-tsukuba \cite{boykov1998markov} & 16 & $6.03 \cdot 10^{-10}$ & $1.17 \cdot 10^{-11}$\\
BVZ-sawtooth \cite{scharstein2002taxonomy} \cite{boykov1998markov}
 & 20 & $9.83 \cdot 10^{-11}$ & $6.11 \cdot 10^{-12}$\\
 BVZ-venus \cite{scharstein2002taxonomy} \cite{boykov1998markov}
 & 22 & $3.40 \cdot 10^{-11}$ & $2.11 \cdot 10^{-12}$\\
 KZ2-tsukuba \cite{kolmogorov2001computing} & 16 & $2.69 \cdot 10^{-10}$ & $1.77 \cdot 10^{-10}$\\
 KZ2-sawtooth \cite{scharstein2002taxonomy} \cite{kolmogorov2001computing}
 & 20 & $4.08 \cdot 10^{-9}$ & $1.56 \cdot 10^{-10}$\\
 KZ2-venus \cite{scharstein2002taxonomy} \cite{kolmogorov2001computing}
 & 22 & $5.21 \cdot 10^{-9}$ & $1.74 \cdot 10^{-10}$\\ \hline
 BL06-camel-sml \cite{lempitsky2006oriented}  & 1 &  \multicolumn{2}{|c|}{$1.21 \cdot 10^{-11}$}\\
BL06-gargoyle-sml \cite{boykov2006photohulls}  & 1 &  \multicolumn{2}{|c|}{$6.29 \cdot 10^{-12}$}\\ \hline
 LB07-bunny-sml \cite{lempitsky2007global} & 1 &  \multicolumn{2}{|c|}{$1.33 \cdot 10^{-10}$}\\
 \hline
\end{tabular}
}
 \caption{Experimental comparison of coordinate-wise minimization on max-flow instances, the references are the original sources of the data and/or to the authors that reformulated these problems as maximum flow. The first 6 rows correspond to stereo problems, the 2 following rows are multiview reconstruction instances, the last row is a shape fitting problem.} \label{ta:table_maxflow}
\end{table}

\subsection{MAP Inference with Potts Potentials}

Coordinate-wise minimization for the dual LP relaxation of MAP inference was intensively studied, see e.g.\ the review \cite{Werner-PAMI07}. One of the formulations is 
\begin{subequations} \label{eq:general_MAP}
\begin{align}
\min \sum_{i \in V} \max_{k \in K} \theta^\delta_i(k) +\sum_{\{i,j\} \in E} \max_{k,l \in K} \theta^\delta_{ij}(k,l) \\
\delta_{ij}(k) \in \mathbb{R} \fspace \forall \{i,j\} \in E, k \in K,
\end{align}
\end{subequations}
where $K$ is the set of labels, $V$ is the set of nodes and $E$ is the set of unoriented edges and
\begin{subequations}
\begin{align}
\theta_i^\delta(k) &= \theta_i(k) - \sum_{j \in N_i}\delta_{ij}(k) \\
\theta_{ij}^\delta(k,l) &= \theta_{ij}(k,l) + \delta_{ij}(k) + \delta_{ji}(l)
\end{align}
\end{subequations}
are equivalent transformations of the potentials. Notice that there are $2 \cdot |E| \cdot |K|$ variables, i.e. two for each direction of an edge. In \cite{Prusa-PAMI-2017}, it is mentioned that in case of Potts interactions, which are given as $\theta_{ij}(k,l) = -\iverson{k \neq l}$, one can add constraints
\begin{subequations}\label{eq:potts_additional_constraint}
\begin{align}
\delta_{ij}(k) + \delta_{ji}(k) &= 0 & \forall & \{i,j\} \in E, k \in K \label{eq:potts_elimination_constraint} \\
-\tfrac12 \leq \delta_{ij}(k) &\leq \tfrac{1}{2} &  \forall & \{i,j\} \in E, k \in K\label{eq:potts_halves}
\end{align}
\end{subequations}
to \eqref{eq:general_MAP} without changing the optimal objective. One can therefore use constraint \eqref{eq:potts_elimination_constraint} to reduce the overall amount of variables by defining
\begin{equation}\label{eq:potts_variables}
\lambda_{ij}(k) = -\delta_{ij}(k) = \delta_{ji}(k)
\end{equation}
subject to $\tfrac12 \leq \lambda_{ij}(k) \leq \tfrac12$. The decision of whether $\delta_{ij}(k)$ or $\delta_{ji}(k)$ should have the inverted sign depends on the chosen orientation $E'$ of the originally undirected edges $E$ and is arbitrary. Also, given values $\delta$ satisfying \eqref{eq:potts_additional_constraint}, it holds for any edge $\{i,j\} \in E$ and pair of labels $k,l \in K$ that $\max\limits_{k,l \in K}\theta_{ij}^\delta(k,l) = 0$, which can be seen from the properties of the Potts interactions.

Therefore, one can reformulate \eqref{eq:general_MAP} into
\begin{subequations}\label{eq:potts}
\begin{align}
\min \sum_{i \in V} & \max_{k \in K} \, \theta^\lambda_i(k)\\
- \tfrac{1}{2} \leq \lambda_{ij}(k) & \leq \tfrac{1}{2} \fspace \forall (i,j) \in E', k \in K,
\end{align}
\end{subequations}
where the equivalent transformation in $\lambda$ variables is given by
\begin{equation}
\theta_i^\lambda(k) = \theta_i(k) + \sum_{(i,j) \in E'} \lambda_{ij}(k) - \sum_{(j,i) \in E'} \lambda_{ji}(k)
\end{equation}
and we optimize over $|E'|\cdot|K|$ variables $\lambda$, the graph $(V,E')$ is the same as graph $(V,E)$ except that each edge becomes oriented (in arbitrary direction). The way of obtaining an optimal solution to \eqref{eq:general_MAP} from an optimal solution of \eqref{eq:potts} is given by \eqref{eq:potts_variables} and depends on the chosen orientation of the edges in $E'$. Also observe that $\theta_i^\delta(k) = \theta_i^\lambda(k)$ for any node $i \in V$ and label $k \in K$ and therefore the optimal values will be equal. This reformulation therefore maps global optima of \eqref{eq:potts} to global optima of \eqref{eq:general_MAP}. However, it does not map interior local minima of \eqref{eq:potts} to interior local minima of \eqref{eq:general_MAP} when $|K| \geq 3$, an example of such case is shown in Appendix \ref{ap:example_potts}.

In problems with two labels ($K = \{1,2\}$), problem \eqref{eq:potts} is subsumed by \eqref{eq:primal_simplified} and satisfies the conditions imposed by Theorem \ref{th:theorem_optimality} because one can rewrite the criterion by observing that
\begin{equation}
\max_{k\in \{1,2\}} \theta^\lambda_i(k) = \max\{\theta^\lambda_i(1)-\theta^\lambda_i(2),0\}+\theta^\lambda_i(2)
\end{equation}
and each $\lambda_{ij}(k)$ is present only in $\theta_i^\lambda(k)$ and $\theta_j^\lambda(k)$. Thus, $\lambda_{ij}(k)$ will have non-zero coefficient in the matrix $B$ only on columns $i$ and $j$. The coefficients of the variables in the criterion are only $\{-1,0,1\}$ and the other conditions are straightforward.

We reported the experiments on the Potts problem in~\cite{CVPR_submission}
where the optimality was not proven yet. In addition, the update
designed in \cite{CVPR_submission} {\em ad hoc\/} becomes just a
special case of our general update here.

\subsection{Binarized Monotone Linear Programs}
\label{sec:monotone}

In \cite{hochbaum2002solving}, integer linear programs with at most two variables per constraint were discussed. It was also allowed to have 3 variables in some constraints if one of the variables occurred only in this constraint and in the objective function. Although the objective function in \cite{hochbaum2002solving} was allowed to be more general, we will restrict ourselves to linear criterion function. It was also shown that such problems can be transformed into binarized monotone constraints over binary variables by introducing additional variables whose amount is defined by the bounds of the original variables, such optimization problem reads
\begin{subequations}\label{eq:binarized_monotone_LP}
\begin{align}
\min \, {w}^Tx&+{e}^Tz \\
Ax - Iz &\leq 0\\
Cx & \leq 0 \\
x &\in \{0,1\}^{n_1}\\
z &\in \{0,1\}^{n_2},
\end{align}
\end{subequations}
where $A,C$ contain exactly one $-1$ per row and exactly one $1$ per
row and all other entries are zero, $I$ is the identity matrix. We
refer the reader to \cite{hochbaum2002solving} for details, where it
is also explained that the LP relaxation of~\eqref{eq:binarized_monotone_LP}
can be solved by min-$st$-cut on an associated graph. We can notice
that the LP relaxation of \eqref{eq:binarized_monotone_LP} is subsumed
by the dual \eqref{eq:LP}, because one can change the minimization into
maximization by changing the signs in $w,e$. Also, the relaxation
satisfies the conditions given by Theorem \ref{th:theorem_optimality}.

In the paper \cite{hochbaum2002solving}, there are listed many problems which are transformable to \eqref{eq:binarized_monotone_LP} and are also directly (without any complicated transformation) subsumed by the dual \eqref{eq:LP} and satisfy Theorem \ref{th:theorem_optimality}, for example, minimizing the sum of weighted completion times of precedence-constrained jobs (ISLO formulation in \cite{chudak1999half}), generalized independent set (forest harvesting problem in  \cite{hochbaum1997forest}), generalized vertex cover \cite{hochbaum2000approximating}, clique problem \cite{hochbaum2000approximating}, Min-SAT (introduced in \cite{kohli1994minimum}, LP formulation in \cite{hochbaum2002solving}).

For each of these problems, it is easy to verify the conditions of Theorem \ref{th:theorem_optimality}, because they contain at most two variables per constraint and if a constraint contains a third variable, then it is the only occurrence of this variable and the coefficients of the variables in the constraints are from the set $\{-1,0,1\}$.

The transformation presented in
\cite{hochbaum2002solving} can be applied to partial Max-SAT and
vertex cover to obtain a problem in the
form~\eqref{eq:binarized_monotone_LP} and solve its LP relaxation. But
this step is unnecessary when applying the presented coordinate-wise
minimization approach.

\section{Concluding Remarks}

We have presented a new class of linear programs that are exactly
solved by coordinate-wise minimization. We have shown that dual LP
relaxations of several well-known combinatorial optimization problems
(partial Max-2SAT, vertex cover, minimum $st$-cut, MAP inference with
Potts potentials and two labels, and other problems) belong, possibly after a
reformulation, to this class. We have shown experimentally (in this
paper and in \cite{CVPR_submission}) that the resulting methods are
reasonably efficient for large-scale instances of these problems. When
the assumptions of Theorem~\ref{th:theorem_optimality} are relaxed
(e.g., general Max-SAT instead of Max-2SAT, or the Potts problem
with any number of labels), the method experimentally still provides
good local (though not global in general) minima.

We must admit, though, that the practical impact of
Theorem~\ref{th:theorem_optimality} is limited because the presented
dual LP relaxations satisfying its assumptions can be efficiently
solved also by other approaches. Thus, max-flow/min-$st$-cut can be
solved (besides well-known combinatorial algorithms such as
Ford-Fulkerson) by message-passing methods such as TRW-S. Similarly,
the Potts problem with two labels is tractable and can be reduced to
max-flow. In general, all considered LP relaxations can be reduced to
max-flow, as noted in~\S\ref{sec:monotone}. Note, however, that this
does not make our result trivial because (as noted
in~\S\ref{sec:reforms}) equivalent reformulations of problems may not
preserve interior local minima and thus message-passing methods are
not equivalent in any obvious way to our method.

It is open whether there are practically interesting classes of linear
programs that are solved exactly (or at least with constant
approximation ratio) by (block-)coordinate minimization and are not
solvable by known combinatorial algorithms such as max-flow. Another
interesting question is which reformulations in general preserve
interior local minima and which do not.

Our approach can pave the way to new efficient large-scale
optimization methods in the future. Certain features of our results
give us hope here. For instance, our approach has an important novel
feature over message-passing methods: it applies to a {\em
constrained\/} convex problem (the box
constraints~\eqref{eq:primal_simplified:box1}
and~\eqref{eq:primal_simplified:box2}). This can open the way to a new
class of applications. Furthermore, updates along large variable
blocks (which we have not explored) can speed algorithms considerably,
e.g., TRW-S uses updates along subtrees of a graphical model, while
max-sum diffusion uses updates along single variables.
%
%
%
\bibliographystyle{splncs04}
\bibliography{mybibliography}
%


\appendix

\section{Details on Coordinate-wise Updates}\label{ap:details_updates}

We now describe coordinate-wise minimization for
problem~\eqref{eq:primal_simplified}, satisfying the relative interior
rule. Objective function~\eqref{eq:primal_simplified_criterion} restricted to a single variable $\varphi_i$ for chosen $i \in [m]$ reads (up to a constant)
\begin{equation}\label{eq:restriction_to_varphi_c}
\max\{w_i-\varphi_i,0\} + \sum_{\substack{j \in [p]\\A_{ij} \neq 0}}\max\{A_{ij}\varphi_i + \kcivarphi, 0\} + a_i\varphi_i,
\end{equation}
where
\begin{equation}\label{eq:def_k_varphi}
\kcivarphi = v_j + \sum_{\substack{i' \in [m]\\i' \neq i}} A_{i'j}\varphi_{i'} + B_{:j}^T\lambda.
\end{equation}
This is a convex piecewise-affine function of $\varphi_i$. Its breakpoints are $w_i$ and $-\kcivarphi/A_{ij}$ for each $j \in [p], A_{ij} \neq 0$. To find its minimum subject to $\lbvarphi_i \leq \varphi_i \leq \ubvarphi_i$, it is enough to consider the cases listed below.
\begin{enumerate}
\item If function \eqref{eq:restriction_to_varphi_c} is strictly decreasing and $\ubvarphi_i$ is finite, then $\ubvarphi_i$ is the unique minimum.
\item If function \eqref{eq:restriction_to_varphi_c} is strictly increasing and $\lbvarphi_i$ is finite, then $\lbvarphi_i$ is the unique minimum.
\item If function \eqref{eq:restriction_to_varphi_c} has an (possibly unbounded) interval $[b_1, b_2]$, where $b_1 \leq b_2$, as its set of minimizers, then the set of minimizers subject to $\lbvarphi_i \leq \varphi_i \leq \ubvarphi_i$ is the projection of $[b_1, b_2]$ onto $[\lbvarphi_i, \ubvarphi_i]$, i.e. an interval $[h_{[\lbvarphi_i, \ubvarphi_i]}(b_1), h_{[\lbvarphi_i, \ubvarphi_i]}(b_2)]$.
\end{enumerate}

In order to perform an update to the relative interior of optimizers, we can simply set $\varphi_i := \ubvarphi_i$ in the first case, $\varphi_i := \lbvarphi_i$ in the second case. For the third case, the update to the relative interior corresponds to setting $\varphi_i$ to some value from $\text{ri}[h_{[\lbvarphi, \ubvarphi]}(b_1), h_{[\lbvarphi, \ubvarphi]}(b_2)]$. In our implementation, we choose the midpoint of this interval if it is bounded. If it is unbounded in some direction, we choose a value in a fixed distance from its finite bound.

To identify which case occurred, one should analyse the slopes of the function between its breakpoints and the region of optima corresponds to the interval where the function \eqref{eq:restriction_to_varphi_c} is constant. If there is no such interval, then its (unrestricted) minimum is at a breakpoint where the function changes from decreasing to increasing or the function is strictly monotone.

In other cases, function \eqref{eq:restriction_to_varphi_c} is unbounded and therefore also the original problem \eqref{eq:primal_simplified} is unbounded.

Objective function \eqref{eq:primal_simplified_criterion} restricted to a single variable $\lambda_i$ reads (up to a constant)
\begin{equation}\label{eq:restriction_to_lambda_c}
\sum_{\substack{j \in [p]\\B_{ij} \neq 0}}\max\{B_{ij}\lambda_i + \kcilambda, 0\} + b_i\lambda_i,
\end{equation}
where
\begin{equation}\label{eq:def_k_lambda}
\kcilambda = v_j + A_{:j}^T\varphi+ \sum_{\substack{i' \in [n]\\i' \neq i}} B_{i'j}\lambda_{c'}.
\end{equation}
To find the minimum of this function subject to $\lblambda_i \leq \lambda_i \leq \ublambda_i$, one can apply the same procedure as with $\varphi_i$, except that the breakpoints will be only $-\kcilambda/B_{ij}$ for each $j \in [p], B_{ij} \neq 0$.

\section{Proof of Theorem \ref{th:theorem_dual_solution}}\label{ap:proof_theorem}

\begin{observation} \label{ob:x_redefinition_a}
For a given $j \in [p]$, $i \in [m]$ with $A_{ij} \neq 0$, fixed value of $\varphi_i$ and the corresponding breakpoint $b = -{\kcivarphi}/{A_{ij}}$ of the restricted criterion \eqref{eq:restriction_to_varphi_c}, the value $x_j$ determined by \eqref{eq:dual_solution} satisfies
$$x_j = 
  \begin{cases}
    \frac12 (1+A_{ij}) & \text{if } \varphi_i  > b \\
    \frac{1}{2} & \text{if } \varphi_i = b \\
    \frac12(1-A_{ij}) & \text{if } \varphi_i < b
  \end{cases}.$$
\end{observation}

\begin{observation}\label{ob:x_redefinition_b}
For a given $j \in [p]$, $i \in [n]$ with $B_{ij} \neq 0$, fixed value of $\lambda_i$ and the corresponding breakpoint $b = -{\kcilambda}/{B_{ij}}$ of the restricted criterion \eqref{eq:restriction_to_lambda_c}, the value $x_j$ determined by \eqref{eq:dual_solution} satisfies
$$x_j = 
  \begin{cases}
    \frac12(1+B_{ij}) & \text{if } \lambda_i  > b \\
    \frac{1}{2} & \text{if } \lambda_i = b \\
    \frac12(1-B_{ij}) & \text{if } \lambda_i < b
  \end{cases}.$$
\end{observation}

\begin{proof}
Both observations follow directly from the definitions of $\kcivarphi$ in \eqref{eq:def_k_varphi} (resp. $\kcilambda$ in \eqref{eq:def_k_lambda}) and the definition of $x$ in \eqref{eq:dual_solution}.
\end{proof}

\begin{observation}
Given an interior local minimum $(\varphi, \lambda)$ of \eqref{eq:primal_simplified}, we can obtain the value of $x_j, j\in [p]$ by finding an $i \in [m]$ (resp. $[n]$) for which the dual constraint \eqref{eq:LP_c7} (resp.  \eqref{eq:LP_c8}) contains $x_j$, i.e. $A_{ij} \neq 0$ (resp. $B_{ij} \neq 0$) and then comparing the position of the breakpoint $-{\kcivarphi}/{A_{ij}}$ with $\varphi_i$ while also considering $\text{sign}(A_{ij})$ (resp. the breakpoint $-{\kcilambda}/{B_{ij}}$ with $\lambda_i$ while also considering $\text{sign}(B_{ij})$).
\end{observation}
\begin{proof}
Follows from Observations \ref{ob:x_redefinition_a}, \ref{ob:x_redefinition_b}.
\end{proof}

\begin{observation}\label{ob:varphi_comparison}
The values of the dual variables $s_i,y_i,z_i$ for $i \in [m]$ and the values of $x_j$ for all $j \in [p]$ with $A_{ij} \neq 0$ can be determined by comparing $\varphi_i$ with $\lbvarphi_i,\ubvarphi_i,w_i$, and the corresponding breakpoints of \eqref{eq:restriction_to_varphi_c}.
\end{observation}
\begin{proof}
Follows from the previous observations and the definition of the dual solution \eqref{eq:dual_solution}.
\end{proof}

\begin{observation}\label{ob:lambda_comparison}
The values of the dual variables $r_i,q_i$ for $i \in [n]$ and the values of $x_j$ for all $j \in [p]$ with $B_{ij} \neq 0$ can be determined by comparing $\lambda_i$ with $\lblambda_i,\ublambda_i$, and the corresponding breakpoints of \eqref{eq:restriction_to_lambda_c}.
\end{observation}
\begin{proof}
Follows from the previous observations and the definition of the dual solution \eqref{eq:dual_solution}.
\end{proof}

\begin{consequence}\label{co:independence}
The validity of the dual constraints \eqref{eq:LP_c7} (resp.  \eqref{eq:LP_c8}) can be checked independently based on an interior local optimum $(\varphi,\lambda)$ by analyzing the individual restrictions \eqref{eq:restriction_to_varphi_c} (resp. \eqref{eq:restriction_to_lambda_c}) of the criterion function \eqref{eq:primal_simplified_criterion}.
\end{consequence}

For the purposes of the following lemma, we define $M(X,f)$ to be the set of minimizers of a convex function $f$ on a closed convex set $X$.

\begin{lemma}\label{lm:lemma_breakpoints}
For a piecewise-affine function
\begin{equation}
f(x) = \sum_{i = 1}^n \max\{c_ix + d_i,0\} + kx,
\end{equation}
with $c_i \in \{-1,1\}$, if the set $M([l,u],f)$ (for $l \leq u$) has more than one element, then $\{b_1, ..., b_n, l,u\} \cap \mathrm{ri}M([l,u],f) = \emptyset$, where $b_i = -{d_i}/{c_i}$ are the breakpoints of the function.
\end{lemma}
\begin{proof}
Since we assume that $M([l,u],f)$ has more than one element and the function $f$ is convex, it must hold for some $l' < u'$ that $M([l,u],f) = [l',u']$ and the function $f$ is constant on $[l',u']$. Naturally, $l \leq l', u' \leq u$ and therefore $l,u \notin \text{ri}[l',u']$.

We will now show that $b_i \notin \text{ri}[l',u']$ by contradiction, so assume the converse, i.e. for some $i$, $b_i \in \text{ri}[l',u']$, in other words
$l' < b_i < u'$. Define
\begin{subequations}
\begin{align}
S & = \{j \in \{1,..., n\} \mid b_j = b_i\} \\
M & = \{b_j \mid j \in \{1,...,n\}-S\} \cup \{l',u'\} \\
\epsilon & = \min\{|b_i-y| \mid y \in M\},
\end{align}
\end{subequations}
where $\epsilon > 0$. We are going to analyze the slope (i.e. the value of derivative) of $f$ on intervals $(b_i-\epsilon,b_i)$ and $(b_i,b_i+\epsilon)$. By definition of $\epsilon$, there is no breakpoint on these intervals and therefore the function is differentiable there and its slope is constant. Additionally, because it is the region of optima, the slope should be zero on these intervals. So assume that the slope on the interval $(b_i-\epsilon,b_i)$ is zero. Then, the slope on $(b_i,b_i+\epsilon)$ is equal to
\begin{equation}
\sum_{\substack{j \in S\\c_j > 0}} c_j + \sum_{\substack{j \in S\\c_j < 0}} -c_j = |S| > 0,
\end{equation}
because all the functions $\max\{c_jx+d_j,0\}$ for $j \in S$ changed their slope at $b_i$, which is a contradiction.
\end{proof}
\begin{consequence}
Apply Lemma \ref{lm:lemma_breakpoints} on the restrictions \eqref{eq:restriction_to_varphi_c},\eqref{eq:restriction_to_lambda_c}. If we are given an interior local minimum $(\varphi, \lambda)$ of \eqref{eq:primal_simplified}, then the equations \eqref{eq:dual_solution} define the dual variables uniquely in the sense that if some primal variable has a non-unique minimizer (i.e. the relative interior of optimizers of the restricted criterion subject to the box constraints has more than one element), then the choice of this minimizer does not influence the position of the variable with respect to the corresponding breakpoints and interval boundaries, as mentioned in Observation \ref{ob:varphi_comparison} and \ref{ob:lambda_comparison}.
\end{consequence}

\begin{consequence}\label{co:consequence_independence} To show that the dual constraints \eqref{eq:LP_c7} (resp. \eqref{eq:LP_c8}) hold for dual variables defined by a primal interior local optimum, it is only necessary to analyze them independently\footnote{Follows from Observation \ref{ob:varphi_comparison} and \ref{ob:lambda_comparison}, respectively Consequence \ref{co:independence}.}. We need to show that for any possible ordering of values $\lbvarphi_i,\ubvarphi_i,w_i$ (resp. $\lblambda_i,\ublambda_i$) and the corresponding breakpoints of \eqref{eq:restriction_to_varphi_c} (resp. \eqref{eq:restriction_to_lambda_c}), any values $a_{i}, A_{i:}$ (resp. $b_i, B_{i:}$) satisfying the conditions in Theorem \ref{th:theorem_optimality}, the values of the corresponding dual variables $s_i,y_i,z_i,x_j$ (resp. $r_i,q_i,x_j$) satisfy the dual constraint \eqref{eq:LP_c7} (resp. \eqref{eq:LP_c8}) if $\varphi_i$ (resp. $\lambda_i$) is any\footnote{Follows from Lemma \ref{lm:lemma_breakpoints}.} chosen value in the relative interior of optimizers of \eqref{eq:restriction_to_varphi_c} (resp. \eqref{eq:restriction_to_lambda_c}) subject to the box constraints.
\end{consequence}

We will now prove that the dual constraint \eqref{eq:LP_c7} is satisfied for each $i \in [m]$ separately and similarly with \eqref{eq:LP_c8}. We will distinguish the cases based on the constraint and the value of $a_c$, resp. $b_i$.

\begin{fact}
For $A$ (resp. $B$) satisfying the conditions of Theorem \ref{th:theorem_optimality} and any $x \in~[0,1]^{p}$, the value of $A_{i:}^Tx$ (resp. $B_{i:}^Tx$) is in range $[-2,2]$.
\end{fact}

\subsection{Proofs for $a_i,b_i$ with large absolute values}

\paragraph{Proof for \eqref{eq:LP_c7} with $a_i > 3$.} If $a_i > 3$, then the function \eqref{eq:restriction_to_varphi_c} is strictly increasing and the bound $\lbvarphi_i$ is necessarily finite, because otherwise the criterion would be unbounded. In this case, the optimum is at $\varphi_i = \lbvarphi_i$. Due to $A_{i:}^Tx  \in [-2,2]$ and $s_i \in [0,1]$, it holds that $y_i = a_i + A_{i:}^Tx - s_i > 0$ and $z_i = 0$. Therefore, the dual constraint \eqref{eq:LP_c7} is satisfied.

\paragraph{Proof for \eqref{eq:LP_c7} with $a_i < -2$.} If $a_i < -2$, then the function \eqref{eq:restriction_to_varphi_c} is strictly decreasing and the bound $\ubvarphi_i$ is necessarily finite, because otherwise the criterion would be unbounded. In this case, the optimum is at $\varphi_i = \ubvarphi_i$. Due to $A_{i:}^Tx  \in [-2,2]$ and $s_i \in [0,1]$, it holds that $a_i + A_{i:}^Tx - s_i < 0$, so $y_i = 0$ and $z_i = a_i + A_{i:}^Tx - s_i - y_i < 0$. Therefore, the dual constraint \eqref{eq:LP_c7} is satisfied.

\paragraph{Proof for \eqref{eq:LP_c8} with $b_i > 2$.} If $b_i > 2$, then the function \eqref{eq:restriction_to_lambda_c} is strictly increasing and the bound $\lblambda_i$ is necessarily finite, because otherwise the criterion would be unbounded. In this case, the optimum is at $\lambda_i = \lblambda_i$. Due to $B_{i:}^Tx  \in [-2,2]$, it holds that $q_i = b_i + B_{i:}^Tx > 0$ and $r_i = 0$. Therefore, the dual constraint \eqref{eq:LP_c8} is satisfied.

\paragraph{Proof for \eqref{eq:LP_c8} with $b_i < -2$.} If $b_i < -2$, then the function \eqref{eq:restriction_to_lambda_c} is strictly decreasing and the bound $\ublambda_i$ is necessarily finite, because otherwise the criterion would be unbounded. In this case, the optimum is at $\lambda_i = \ublambda_i$. Due to $B_{i:}^Tx  \in [-2,2]$, it holds that $r_i = b_i + B_{i:}^Tx < 0$ and $q_i = 0$. Therefore, the dual constraint \eqref{eq:LP_c8} is satisfied.

\subsection{Proof for the remaining values}

We will now focus on satisfaction of a single dual constraint \eqref{eq:LP_c7} with $a_i \in [-2,3] \cap \mathbb{Z}$ as given by Consequence \ref{co:consequence_independence} (the analysis for \eqref{eq:LP_c8} with $b_i \in [-2,2] \cap \mathbb{Z}$ would be similar). By the conditions of Theorem \ref{th:theorem_optimality}, we can assume that $A_{i:}$ has at most two non-zero entries, which will be WLOG on positions $A_{i1},A_{i2}$. It now only remains to show that for any values $A_{i1}, A_{i2}, a_i$, and any ordering\footnote{By ordering, we mean comparisons by strict inequality and equality, not only w.r.t $\leq$.} of values $\lbvarphi_i, \ubvarphi_i, w_i$, and the corresponding breakpoints of \eqref{eq:restriction_to_varphi_c}, a value of $\varphi_i$ in the relative interior of optimizers of \eqref{eq:restriction_to_varphi_c} subject to $\lbvarphi_i \leq \varphi_i \leq \ubvarphi_i$ will define an assignment \eqref{eq:dual_solution} which is feasible for \eqref{eq:LP_c7}.

If $A_{i1},A_{i2} \neq 0$, then there are two breakpoints $b_1,b_2$ (if some of the coefficients was zero, then the corresponding breakpoint would not exist, but otherwise the analysis would be analogous). There is only a finite amount of possible orderings of values $\lbvarphi_i,\ubvarphi_i,w_i,b_1,b_2$, because there is at most $5!$ orderings w.r.t. $\leq$, i.e.
\begin{subequations}
\begin{align}
\lbvarphi_i \leq \ubvarphi_i \leq &w_i \leq b_1 \leq b_2 \label{eq:ordering1}\\
\lbvarphi_i \leq \ubvarphi_i \leq &w_i \leq b_2 \leq b_1\\
& \vdots \nonumber \\
b_2 \leq b_1 \leq &w_i \leq \ubvarphi_i \leq \lbvarphi_i,
\end{align}
\end{subequations}
some of which are not allowed due to $\lbvarphi_i < \ubvarphi_i$. For each of them, there is $2^4$ options to diversify $=$ and $<$, for example for ordering \eqref{eq:ordering1}:
\begin{subequations}
\begin{align}
\lbvarphi_i < \ubvarphi_i < &w_i < b_1 < b_2\\
\lbvarphi_i < \ubvarphi_i < &w_i < b_1 = b_2\\
\lbvarphi_i < \ubvarphi_i < &w_i = b_1 < b_2\\
\lbvarphi_i < \ubvarphi_i < &w_i = b_1 = b_2\\
& \vdots \nonumber
\end{align}
\end{subequations}
Therefore, we can perform case analysis based on the given ordering and values $A_{i1},A_{i2} \in \{-1,0,1\}, a_i \in \{-2,-1,0,1,2,3\}$. For each of the (at most) $3^2 \cdot 6 \cdot 2^4 \cdot 5!$ cases, we will decide the conditions on the position of $\varphi_i$ and based on the position, set the values $s_i,z_i,y_i,x_1,x_2$ and check that the dual constraint \eqref{eq:LP_c7} is satisfied.

Due to large amount of considered cases, the proof is automated and attached in the supplementary material (Matlab script $\texttt{automated\_proof.m}$). In the automated proof, we use the fact that for each of the $2^4 \cdot 5!$ orderings, there exists an assignment to values $\lbvarphi_i,\ubvarphi_i,w_i,b_1,b_2$ from the set $\{1, ..., 5\}$ such that the ordering (including $<$ and $=$ relations) is preserved. For each such assignment and each value of $A_{i1},A_{i2},a_i$, we find the region of optima, assign a value to $\varphi_i$ such that it is in the relative interior of the region of optima, calculate the values of the dual variables and check their validity.

A similar procedure is performed in order to prove satisfaction of the dual constraint \eqref{eq:LP_c8}. This proves the feasibility of the solution \eqref{eq:dual_solution}, complementary slackness can be trivially seen from the definition of the solution and comparing the corresponding constraints/variables in \eqref{eq:LP}.

\section{Example for Weighted Vertex Cover}\label{ap:example_weighted_vertex_cover}

Consider the graph $K_{1,n}$, i.e. a complete bipartite graph with one node $x_1$ on one side with weight equal to $n-\frac12$ and $n$ nodes $x_2, ..., x_{n+1}$ on the other side, which have weights equal to $1$. Then, $x = (0,1,1,1,...,1)$ is an interior local minimum of \eqref{eq:vertex_cover} with respect to all blocks of variables of size $n$. Let us now consider such block $B$.
\begin{itemize}
\setlength\itemsep{0em}
\item If the variable $x_1$ is not in block $B$, no update is possible due to the constraints on edges.
\item If the variable $x_1$ is in the block, then we could update all $x_i$ for $i \in B, i \neq 1$ to $x_i - \epsilon$ and $x_1$ to $x_1 + \epsilon$ for $\epsilon \in [0,1]$ and the values of variables will remain feasible\footnote{We could also increase $x_1$ to $x_1 + \epsilon + \delta$ for $\delta \in [0,\epsilon-1]$, but such update would never be optimal for any $\delta > 0$.}. This update would change the criterion by $\epsilon (n - \frac12)- \epsilon(n-1) = \frac12 \epsilon \geq 0$ and since we are minimizing, the optimum update is for $\epsilon = 0$, i.e. remain with the current value of $x$.
\end{itemize}
Thus, the shown point is an interior local minimum w.r.t. all blocks of variables of size $n = |V|-1$, but has worse criterion value than the global optimizer $x = (1, 0, 0, ..., 0)$.

\section{Example for the Potts Problem} \label{ap:example_potts}

Consider the case with $|K| = 3$ labels and a chain graph with 4 nodes. We can see the numerical example in Figure \ref{fig:potts}, all the active (maximal) labels in nodes are shown as black and the inactive as white and their transformed values are shown under them (the first number in the formula is their original value). The values of the $\lambda$ variables are on the corresponding edges and the orientation of the edges is from left to right. One can clearly see that each of the $\lambda$ variables is in its relative interior of optima.

If we transform the $\lambda$ variables into the general form of MAP estimation \eqref{eq:general_MAP} with $\delta$ variables, then the result is in Figure \ref{fig:diffusion}. Clearly, the unary potentials $\theta^\delta_i(k)$ did not change by the transformation. The binary potentials $\theta^\delta_{ij}(k,l)$ have values $-3,-2,-1,0$, depending on the corresponding $\delta_{ij}(k)$ and $\delta_{ji}(l)$. One can observe that this setting of variables is not an interior local minimum, because by reasoning from [29], the arc consistency closure of the maximal nodes and maximal edges in Figure \ref{fig:diffusion} is empty and therefore we can decrease the criterion by Max-Sum diffusion (i.e. by coordinate-wise updates into the relative interior).

At the initial stage, $\delta_{12}(1)$ is not in the relative interior of optimizers, but on its boundary, we will update it. For the same reason, we will then sequentially update $\delta_{21}(1)$, $\delta_{23}(1)$, $\delta_{32}(1)$ and $\delta_{34}(1)$. Then, $\delta_{43}(1)$ would not be the optimal choice and its change would decrease the criterion.

\begin{figure}
   \centering
   \includegraphics[width=\columnwidth,trim={2cm 3cm 2cm 3cm},clip]{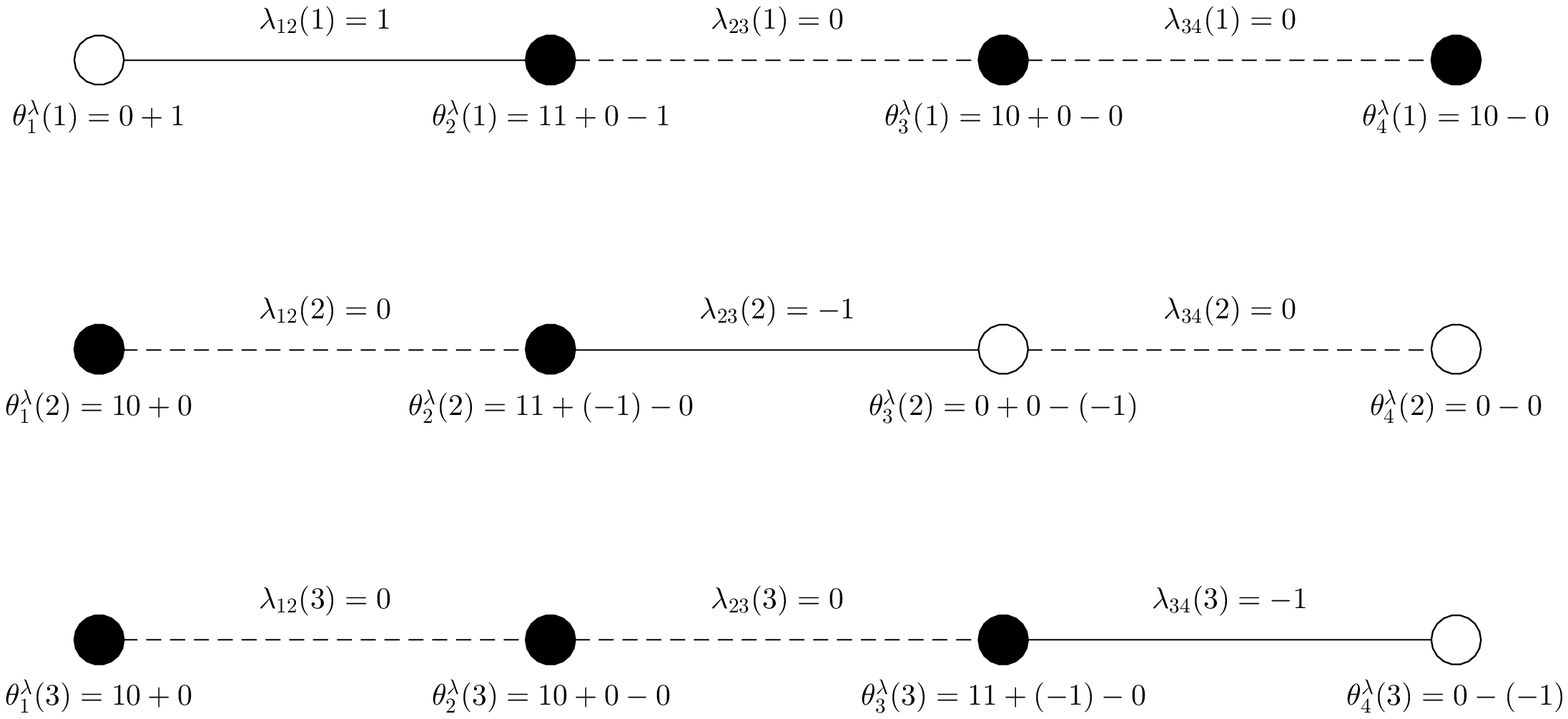}
   \caption{An example of an interior local minimum of the Potts problem,  non-zero $\lambda$ variables are denoted by solid lines and zero $\lambda$ variables are denoted by dashed lines. All values should be halved, which was omitted for clarity and better reading.}
   \label{fig:potts}

\bigskip
   \includegraphics[width=\columnwidth,trim={2cm 3cm 2cm 3cm},clip]{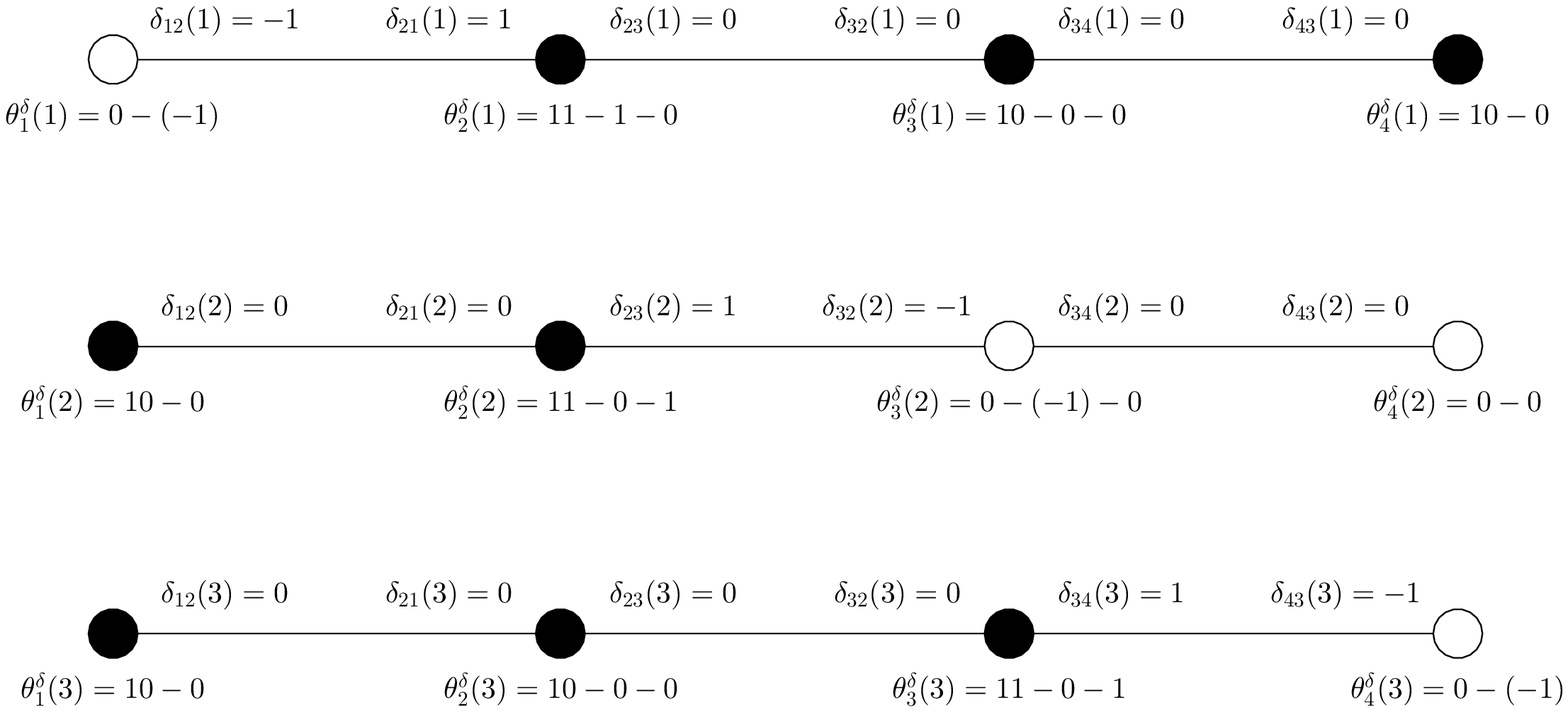}
   \caption{Corresponding problem in variables of Max-Sum diffusion, maximal values of $\theta^\delta_{ij}(k,l)$ are shown by solid lines in each edge, non-maximal values $\theta^\delta_{ij}(k,l)$ are not drawn. All values should be halved, which was omitted for clarity and better reading.}
   \label{fig:diffusion}
\end{figure}

\end{document}